\documentclass[10pt,a4paper]{article}
\usepackage[margin=1.95cm]{geometry}
\usepackage{amsmath,amsthm,amsfonts,amssymb,amscd,cite,graphicx,float}
\usepackage{caption,subcaption}
\usepackage{times}
\usepackage{amssymb,amsthm,latexsym,amsmath,epsfig,pgf}
\usepackage{lineno}
\usepackage{lineno,hyperref}
\usepackage{lipsum,fancyhdr,datetime}
\usepackage{wrapfig}
\usepackage{lscape}
\usepackage{rotating}
\usepackage{epstopdf}
\newtheorem{theorem}{Theorem}[section]
\newtheorem{lemma}{Lemma}[section]
\newtheorem{definition}{Definition}[section]

\theoremstyle{remark}

\theoremstyle{remark}

\modulolinenumbers[1]

\begin{document}
\begin{center}
 {\large \bf Odd sun-free Triangulated Graphs are $S$-perfect}   

\noindent
\\

\noindent
G. Ravindra, Sanghita Ghosh
, Abraham V. M.\\

\noindent
\footnotesize {\it Department of Mathematics, CHRIST (Deemed to be University), Bengaluru, India.}\\
\footnotesize {\it gravindra49@gmail.com, sanghita.ghosh@res.christuniversity.in, frabraham@christuniversity.in}
\noindent
\end{center}

\noindent

\setcounter{page}{1} \thispagestyle{empty}

\baselineskip=0.20in

\normalsize

 \begin{abstract}
For a graph $G$ with the vertex set $V(G)$ and the edge set $E(G)$ and a star subgraph $S$ of $G$, let $\alpha_S(G)$ be the maximum number of vertices in $G$ such that no two of them are in the same star subgraph $S$ and $\theta_S(G)$ be the minimum number of star subgraph $S$ that cover the vertices of $G$. A graph $G$ is called $S$-perfect if for every induced subgraph $H$ of $G$, $\alpha_S(H)=\theta_S(H)$. Motivated by perfect graphs discovered by Berge, Ravindra introduced $S$-perfect graphs. In this paper we prove that a triangulated graph is $S$-perfect if and only if $G$ is odd sun-free. This result leads to a conjecture which if proved is a structural characterization of $S$-perfect graphs in terms of forbidden subgraphs.

 \end{abstract}

\baselineskip=0.20in

\section{Introduction}
\par All graphs $G = (V, E)$ in this paper are finite and simple with the vertex set $V (G)$ and the edge set $E(G)$. Inspired by perfect graphs due to Berge\cite{berge}, the concept of $ F $-perfect graphs was introduced by Ravindra in 2011 \cite{taiwan}. For the terminology and notations which are not defined here, we refer the readers to West \cite{west2017}.
\par A graph $G$ is said to be \textit{triangulated} if $G$ has no induced cycle of length at least 4. If $D\subseteq V(G)$, the subgraph of $G$ induced by $D$ is obtained by deleting the vertices of $V(G)-D$ and is denoted by $G[D]$. Given a graph $H$, we say that a graph $G$ is $H$-free, if $G$ does not contain an induced subgraph isomorphic to $H$. Given a family $\{H_1, H_2, ... \}$ of graphs, we say that $G$ is $(H_1, H_2, ...)$-free if $G$ is $H_i$ – free, for every $i \geq 1$. 
\par A \textit{star} is a tree consisting of one vertex adjacent to all other vertices.
Note that $K_1$ and $K_2$ are stars.
\par A vertex $v$ in $G$ is called an \textit{simplicial vertex} if $v$ belongs to only one maximal clique of $G$. A clique of $G$ is \textit{free} if it contains at least one simplicial vertex. A triangle is \textit{free triangle} if it has only one simplicial vertex.
\par Let $G$ be a graph. An \textit{$\mathbb{S}$}-\textit{cover} of $G$ is a family of stars contained in $G$ such that every vertex of $G$ is in one of the stars. The \textit{S}-\textit{covering number} of $G$ is the minimum cardinality of a $\mathbb{S}$-cover and is denoted by $\theta_S(G)$. A \textit{$\theta_S$-cover} of $G$ is an $\mathbb{S}$-cover of $G$ containing $\theta_S(G)$ stars. 

\par A set $T\subseteq V(G)$ is an \textit{$S$-independent} set of $G$ if no two vertices of $T$ are contained in the same star in $G$ (equivalently, any two vertices in $T$ are at a distance at least 3). The maximum cardinality of an $S$-independent set is called the \textit{$S$-independence number} of $G$ and is denoted by $\alpha_S(G)$. An \textit{$\alpha_S$- independent set} of $G$ is a $S$-independent set with $\alpha_S(G)$ vertices.
\par It is easy to see that the graph parameters $\alpha_S(G)$ and $\theta_S(G)$ both satisfy the additive property, that is
\begin{itemize}
    \item $\alpha_S(\bigcup_{i=1}^{k}G_i)=\sum_{i=1}^{k} \alpha_S(G_i)$, and
    \item $\theta_S(\bigcup_{i=1}^{k}G_i)=\sum_{i=1}^{k} \theta_S(G_i)$.
\end{itemize}
\par It is also interesting to observe that the parameters $\alpha(G)$ and $\theta(G)$ of Berge perfect graphs satisfy monotone property, that is if $H$ is an induced subgraph of $G$, then $\alpha(H)\leq \alpha(G)$ and $\theta(H)\leq\theta(G)$. However, the parameters $\alpha_S(G)$ and $\theta_S(G)$ do not satisfy monotone property. That is, if $H$ is an induced subgraph of $G$, $\alpha_S(H)$ (or $\theta_S(H)$) may exceed $\alpha_S(G)$ (or $\theta_S(G)$). For example if $G=K_{1,n}, n\geq 2$ with central vertex $v_0$, then $\alpha_S(K_{1,n})=1$ and $\theta_S(K_{1,n})=1$. However $\alpha_S(K_{1,n}-v_0)=n$ and $\theta_S(K_{1,n}-v_0)=n, n\geq 2$.
\par For any graph $G$, $\theta_S(G) \geq \alpha_S(G)$. So, if for some graph $H$, $\theta_S(H)\leq\alpha_S(H)$, then $\alpha_S(H)=\theta_S(H)$.
\par Now we examine the graphs with the property that $\alpha_S(H)=\theta_S(H)$ for every induced subgraph $H$ of $G$ and call such graphs $S$-perfect graphs. We formally present the definition of \textit{$S$-perfect graphs}.
\begin{definition}
A graph is called $S$-perfect if $ \theta_{S}(H)=\alpha_{S}(H) $, for every induced subgraph $H$ of $G$.
\end{definition}

\par A graph $G$ is said to be \textit{minimal $S$-imperfect}  if (i) $\theta_S(G) \neq \alpha_S(G)$ and (ii) $G-v$ is $S$-perfect for every vertex $v$ of $G$. Every minimal $S$-imperfect graph is a connected graph, since a graph $G$ is $S$-perfect if and only if every component of $G$ is $S$-perfect. If $G$ is not $S$-perfect, now onwards we may assume that $G$ is minimal $S$-imperfect.
 \par Since $S$-perfectness is a hereditary property, one expects a forbidden subgraph characterization for $S$-perfect graphs. We realize this in this paper and prove a characterization theorem for triangulated $S$-perfect graphs.
\par The property that the parameters $\alpha_S$ and $\theta_S$ do not satisfy monotone property causes some difficulty in some of the results related to their equality.
\par Before we state the characterization theorem for triangulated $S$-perfect graphs, we define \textit{sun graphs}.
 
\begin{definition}\label{extendedsungraph}
Let $H$ be a graph with a Hamiltonian cycle $C=v_1 v_2\dots v_k v_1$. Let $A_i, 1\leq i \leq k$ be mutually disjoint complete graphs such that $|A_i|\geq 1$, $\forall$ $i$ such that $1\leq i \leq k$. Let $H_1$ be a graph constructed from $H$ such that every vertex in $A_i$ is adjacent to $v_i$ and $v_{i+1}$ and all $u_i$s in $A_i$ are simplicial in $H_1$ ($i's$ are taken modulo $k$). $H_1$ is called $k$-extended sun. $H_1$ is an extended odd (even) sun if $k$ is odd (even). If $|A_i|=1$, for all $i$ then $H_1$ is $k$-sun graph.
$H_1$ is an odd (even) sun if $k$ is odd (even).
\end{definition}

Obviously, a sun is an extended sun. A 3-sun is contained in every $k$-extended sun. For example in Figure \ref{fig:sungraphs} we see\linebreak $[\{\{A_1\}, v_1, v_2, v_5, v_4,\{A_5\}\}]$=3-sun.
\begin{figure}[h]
            \centering
            \includegraphics[width=0.65\linewidth]{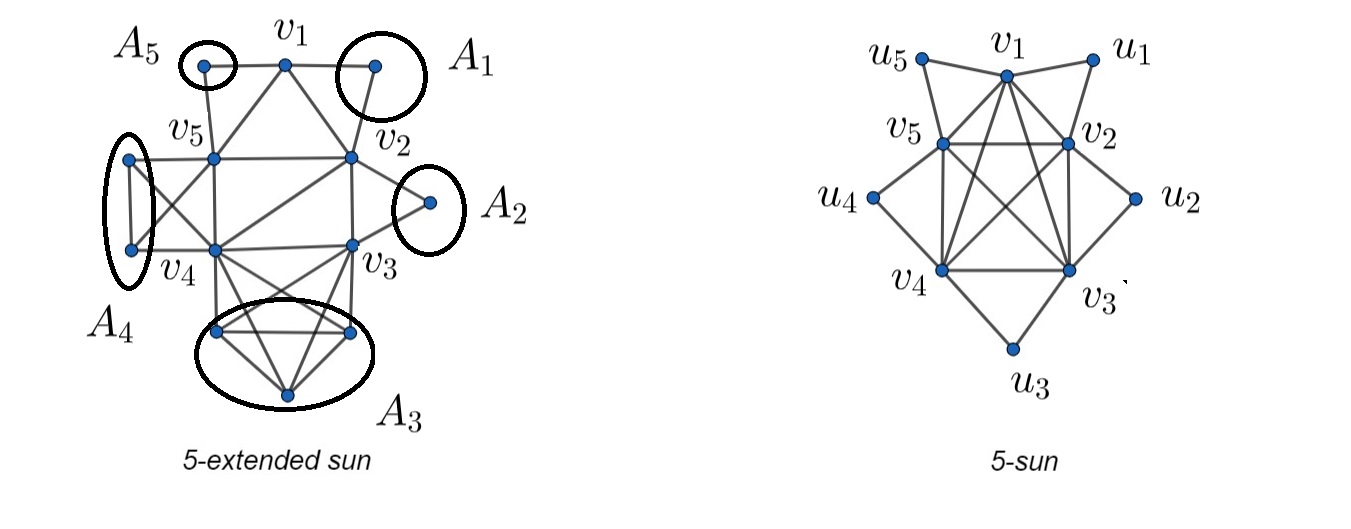}
            \caption{Examples of extended sun and sun on 5 vertices}
            \label{fig:sungraphs}
        \end{figure}
        
 \par The goal of this paper is to prove the following result.      
\newtheorem*{main}{\textbf{Theorem \ref{main}}}
\begin{main}[\textbf{Characterization Theorem for Triangulated $S$-perfect graphs}]
A triangulated graph $G$ is $S$-perfect if and only if $G$ is odd sun-free.
\end{main}
\par The theorem is a min-max theorem for triangulated $S$-perfect graphs.

\par For any graph $G$, the central vertices of stars in $\mathbb{S}$-cover of $G$ is a dominating set of $G$. The minimum parameter $\theta_S(G)$ is essentially the domination number, $\gamma(G)$ of $G$, which has many applications like surveillance, controlling and monitoring. The maximum parameter $\alpha_S(G)$ have been extensively studied in various contexts with different terminology. For example, $S$-independent sets are studied in \cite{colorclass}, where an $S$-independent set in $G$ corresponds to a color class in a $(k,3)$-coloring of $G$. A study on characterization of star-perfect graphs, where all the stars in a star-cover of $G$ are essentially induced, is also done in \cite{starp}. 

\par Our characterization of triangulated $S$-perfect graphs goes parallel to triangulated neighbourhood perfect graphs. The notion of neighbourhood number which was introduced by  Sampathkumar and Neeralagi \cite{nbdno}. Lehel and Tuza defined neighbourhood perfect graphs and characterized triangulated neighbourhood perfect graphs \cite{npg}. Though seemingly neighbourhood perfect graphs and $S$-perfect graphs appear different (For example: $C_{6k+2}$ and $C_{6k+4}, k\geq1$ are neighbourhood perfect but not $S$-perfect and $C_{6k+3},k\geq~1$ are $S$-perfect but not neighbourhood perfect), surprisingly they are same for triangulated graphs (Theorem~\ref{main}).

\section{Results and Discussions}
We use the following theorem and lemmas in proving the main theorem which characterizes triangulated $S$-perfect graphs.
\begin{theorem}\label{sperfect}\cite{starp}
 A graph $G$ is $star$-perfect if and only if $G$ is $(C_3, C_{3k+1}, C_{3k+2})$-free, $k\geq 1$. 
\end{theorem}
\begin{lemma}\label{path}\cite{starp}
Let $k$ be any positive integer. Then,
\begin{itemize}
    \item[(i)] $\theta_s(P_k)=\left\lceil\dfrac{k}{3}\right\rceil$ and $\alpha_s(P_k)=\left\lceil\dfrac{k}{3}\right\rceil$.
    \item[(ii)] $P_k$ is star-perfect.
    \item[(iii)] The disjoint union of paths is a star-perfect graph.
\end{itemize}\end{lemma}
The proofs of this Lemmas \ref{lemma2.2}, \ref{forbidden}, \ref{l3}  are similar to as shown in \cite{starp}.
\begin{lemma}\label{lemma2.2}
Any cycle of length $3k$, $k\geq1$ is $S$-perfect.
\end{lemma}

\begin{lemma}\label{forbidden}
Any cycle of length $3k+1$ or $3k+2$, $k\geq1$ is minimal $S$-imperfect.
\end{lemma}
Our next goal is to show that a minimal $S$-imperfect graph is a block.
\begin{lemma}\label{l3}
If $G$ is minimal $S$-imperfect graph, then $G$ is a block.
\end{lemma}

\begin{lemma}\label{unilemma}\cite{dirac,bgraphs}
Let $G$ be a triangulated graph, then $G$ has a simplicial vertex. 
\end{lemma}
\begin{lemma}\label{anyv}
 If $G$ is minimal $S$-imperfect, then $\theta_S(G)=\alpha_S(G)+1$.
\end{lemma}
\begin{proof}
Since $G$ is triangulated, by Lemma \ref{unilemma}, $G$ has a simplicial vertex $v$. If $\alpha_S(G)=1$, then $\theta_S(G)\neq 1$ since $\theta_S(G)\neq \alpha_S(G)$. Therefore $\theta_S(G)\geq2$. We observe that $\alpha_S(G-v)=1$.\hfill (A)\\
If not, there exist  $v_1, v_2$ in $V(G-v)$ such that $\{v_1, v_2\}$ is $S$-independent in $G-v$. Since $\alpha_S(G)=1$, $v\leftrightarrow \{v_1, v_2\}$. Then $v_1\leftrightarrow v_2$, since $v$ is simplicial in $G$, a contradiction. Since $\theta_S(G)\geq 2$, $\theta_S(G)\geq \alpha_S(G)+1$. Also $\theta_S(G)\leq \theta_S(G-v)+1$, since a $S$-cover of $G$ contains at least $\theta_S(G-v)+1$ stars. Then $\theta_S(G)\leq \alpha_S(G-v)+1=\alpha_S(G)+1$, implying $\theta_S(G)=\alpha_S(G)+1$, so the lemma is true if $\alpha_S(G)=1$. If $\alpha_S(G)\geq 2$, let $T=\{v_1, v_2,\dots, v_k\}, k\geq 2$ be an $\alpha_S$-independent set of $G-v$. If $k=1$, then $\alpha_(G-v)=1$ and so $\theta_S(G-v)=1$. Then $\theta_S(G)\leq \theta_S(G-v)+1=\alpha_S(G-v)+1=2$, by (A). That is $\theta_S(G)=\alpha_S(G)$, a contradiction to $G$ being minimal $S$-imperfect. Therefore $k\geq 2$.\\
If $T$ is not an $\mathbb{S}$-independent set in $G$, then then there are two vertices $v_1, v_2$ in $T$ such that $\{v_1,v_2\}$ is not $S$-independent in $G$. As argued earlier we have a contradiction.  $|T|\leq \alpha_S(G)$. By definition of $T$, $|T|=\alpha_S(G-v)$. Therefore $\alpha_S(G)\geq\alpha_S(G-v)$
Since $\alpha_S(G-v)=\theta_S(G-v)$ we have $\alpha_S(G)\geq\theta_S(G-v)$. That is $\alpha_S(G)+1\geq\theta_S(G-v)+1\geq\theta_S(G)$. Therefore $\theta_S(G)=\alpha_S(G)$ or $\theta_S(G)=\alpha_S(G)+1$. However $\theta_S(G)\neq\alpha_S(G)$, since $G$ is minimal $S$-imperfect, therefore $\theta_S(G)=\alpha_S(G)+1$. Hence the lemma. 
\end{proof}
\begin{lemma}\label{neighbor}
    If $G$ is a minimal $S$-imperfect graph, then for $u,v\in V(G)$, neither $N(u)\subseteq N(v)$ nor $N(v)\subseteq N(u)$.
\end{lemma}
\begin{proof}
 Suppose false, then say $N(u)\subseteq N(v)$. Let $S=\{S_1, S_2, \dots S_w, \dots, S_k\}$ be a $\theta_S$-cover of $G$ where each star $S_i$ is maximal and $S_v$ is a star containing $v$. By Lemma \ref{anyv}, the stars in $\theta_S$-cover of $G$ are $\alpha_S(G)+1$ in number. Since $N(u)\subseteq N(v)$, any star $S'$ containing $u$ is a substar (a subgraph which is a star) of $S_v$, $\theta_S(G-u)\leq\theta_S(G)$. We observe that $\theta_S(G-u)=\theta_S(G)$. If $\theta_S(G-u)<\theta_S(G)$, then $\theta_S(G-u)\leq \theta_S(G)-1$, and $G-u$ will be covered by $\theta_S(G)-1$ stars. Since $S'$ is a substar of $S_v$, $G$ will also be covered by $\theta_S(G)-1$ stars, contradiction to $\theta_S(G)$ being minimum. Therefore $\theta_S(G-u)=\theta_S(G)$.
\par Let $T$ be an $\alpha_S$-independent set in $G-u$. We observe that $T$ is an $\alpha_S$-independent set in $G$. If not, $u$ is adjacent to at least two vertices in $T$ in $G$. Thus $v$ is adjacent to two vertices in $T$ since $N(u)\subseteq N(v)$. But then $T$ is not an $S$-independent set in $G-u$ as $v\in V(G-u)$, a contradiction. Thus $T$ is an $S$-independent set in $G$ and $\alpha_S(G)\geq|T|=\alpha_S(G-u)=\theta_S(G-u)=\theta_S(G)$. This implies $\alpha_S(G)\geq\theta_S(G)$, a contradiction. Hence the lemma.
\end{proof}
\begin{lemma}\label{hamiltonian}
If $G$ is a triangulated minimal $S$-imperfect graph, then $G$ is Hamiltonian.
\end{lemma}
\begin{proof}
Let $C=v_1v_2\dots v_k v1$ be a largest cycle in $G$. If $V(C)=V(G)$, then the lemma is true.
\par So let $v\in V(G)-V(C)$. Since $G$ is minimal $S$ imperfect graph, it is a block and hence connected. Let $v\leftrightarrow v_1$. Since $G$ is a block, there exists an induced cycle $C'$ containing the edges $vv_1$ and $v_1v_2$.
\par Since $G$ is triangulated, $C'$ is a triangle. This implies that $v\leftrightarrow v_2$. Therefore $C$ together with $v$ is a bigger cycle than that of $C$, a contradiction to the choice of $C$. Hence such a $v$ does not exist and therefore $G$ is Hamiltonian.
\end{proof}

\begin{lemma}\label{implemma}
    Let $G$ be a triangulated $S$-imperfect graph. Then there exists an extended sun $G^*$ containing $G$ as an induced subgraph.
\end{lemma}
\begin{proof}
Since $G$ triangulated, $G$ has a simplicial vertex $u$, by Lemma \ref{unilemma}. Let $Q$ be a free clique in $G$ containing $u$. The number of non-simplicial vertices in $Q$ is at least 2. If not, then the only non-simplicial vertex in $Q$ is a cut-vertex in $G$, a contradiction to the fact that $G$ is a block, by Lemma \ref{l3}.  
If the simplicial vertices of $G$ are removed, the resulting graph $H$ is obviously Hamiltonian. Let $C=v_1v_2\dots v_k v_1$ be a Hamiltonian cycle in $H$. Let $A_1,A_2,\dots, A_l$, $1\leq l\leq k$ be mutually disjoint complete subgraphs in $G$ such thatfor every $u_i\in A_i$, $u_i\leftrightarrow \{v_i, v_{i+1}\}$ and $u_i$ is simplicial in $G$. If $l=k$, we are done. If $l<k$, let $u_{l+1}$ be a vertex not in $G$ and let $u_{l+1}$ is not adjacent to any simplicial vertex in $G$. Let $G_1$ be the graph formed by $G$  and $u_{l+1}$ such that $u_{l+1}$ is simplicial in $G_1$. If $G_1$ is an extended sun, then the lemma is true, otherwise we repeat the process to get $G_2, G_3, \dots G_m$, where $G_m$ is an extended sun. Considering $G_m=G^*$, $G$ is an induced subgraph of $G^*$ and hence the lemma is true.
\end{proof}
\begin{lemma}\label{lemma2.9}
Let $G$ be a minimal $S$-imperfect graph. Let $u$ be a vertex not in $G$. Let $G'$ be a graph formed by $G$ and $u$ such that $u$ is simplicial in $G'$ and $u$ is adjacent to an edge in $G$ which is not adjacent to a simplicial vertex of $G$. Then $\alpha_S(G')\neq \theta_S(G')$.
\end{lemma}
\begin{proof}

On the contrary, suppose $\alpha_{S}(G')=\theta_{S}G')$. This implies that every vertex of an $\alpha_S-$set of $G'$ is in exactly one star of $\theta_S-$cover of $G'$. Since $u$ is simplicial in $G^{'}$, by nature of $G^{'}$ any $S-$independent set in $G'$ cannot have more than $\alpha_S(G)+1$ vertices. By Lemma \ref{mainlemma}, $\theta_S(G)=\alpha_S(G)+1$. Then $\alpha_S(G)+1=\theta_S(G)\leq \theta_S(G')=\alpha_S(G')$, since a $\theta_S$-cover of $G'$ contains a $\theta_S$-cover of $G$. If $\alpha_S(G)+1=\alpha_S(G')$, then $\theta_S(G)=\theta_S(G')$. Then there is an $S-$independent set of $G$ of size $\alpha_S(G)+1$, a contradiction to the fact that $\alpha_S(G)$ is maximum. Therefore $\alpha_S(G')\neq\theta_S(G')$.
\end{proof}

\begin{lemma}\label{oddsunlemma}
Odd sun is not $S$-perfect.
\end{lemma}
\begin{proof}
Let $G$ be an odd sun and $U=\{u_1, u_2,\dots, u_k\}$ be the set of simplicial vertices and $W=\{v_1, v_2, \dots, v_k\}$ be the set of non-simplicial vertices in $G$. Then $T^{'}=\{u_1, u_3,\dots,u_{k-2}\}$ forms an $S$-independent set of $G$ and the stars centered at $v_1, v_3,\dots, v_{k}$ form an $\mathbb{S}$-cover of $G$, say $\mathbb{S^{'}}$. Then this implies that $|\mathbb{S^{'}}|=|T^{'}|+1$ implying $\alpha_S(H^{*})\neq \theta_S(H^{*})$, hence the lemma.
\end{proof}
Since $k$-odd sun is an induced subgraph of $k$-extended sun, then the following lemma 
is immediate.
\begin{lemma}\label{oddextendedsun}
Odd extended sun is not $S$-perfect.
\end{lemma}
\begin{definition}
Let $P$ be a path. A special path $P^*$ is constructed from $P$ such that an edge of $P$ is contained in a free triangle.
\end{definition}
\begin{lemma}\label{specialpath}
A special path is $S$-perfect.
\end{lemma}
\begin{proof}
Let $P^*$ be a special path formed from the path $P$. If $P^*$ is $K_2$ or $K_3$, then obviously $P$ is $S$-perfect. So let $P^*\neq K_2$ or $K_3$. Then $P^*$ will have a cut vertex, since every block of $P^*$ is $K_2$ or $K_3$, by definition.If $P^*$ is not $S$-perfect, let $P^*$ is minimal $S$-imperfect. By Lemma \ref{l3}, $P^*$ is a block, a contradiction.
\end{proof}

\begin{lemma}\label{evensunlemma}
		If $G$ is a 3-sun free triangulated even extended sun, then $G$ is $S$-perfect.
	\end{lemma}
	\begin{proof}
	Let $U=\{u_1, u_2,\dots, u_k\}$ be the set of simplicial vertices and $W=\{v_1, v_2, \dots, v_k\}$ be the set of non-simplicial vertices in $G$. Let $u_i$ represent the simplicial vertices in $A_i$. $A_i$ and $C$ have the same meaning as in Definition \ref{extendedsungraph}. If $v_iv_j$ and $v_{i+1}v_{j+2}$ are alternate edges in $C$, $u_i$ and $u_{i+2}$ are at a distance 3. Since $k$ is even, $\{u_1, u_3, \dots, u_l, u_{l+3}, \dots\}$ is an $S$-independent set in $G$ of size $\dfrac{k}{2}$. Similarly, the stars at $v_1, v_3,\dots, v_l, v_{l+3}$ is an $S$-cover of $G$ of size $\dfrac{k}{2}$. Then $\alpha_S(G)\geq \dfrac{k}{2} \geq \theta_S(G)$. This implies that $\alpha_S(G)=\theta_S(G)$. Every proper induced subgraph of $G$ is a complete graph or a union of disjoint paths or special paths. Every complete graph is $S$-perfect and by Lemma \ref{path} and Lemma \ref{specialpath}, every induced subgraphs of $G$ is $S$-perfect.
	\end{proof}

\begin{lemma}\label{mainlemma}
Let $G$ be minimal $S$-imperfect triangulated graph. If for $v, u,w\in V(G)$ such that $v$ is a simplicial vertex in $G$, $u,w\in N(v)$, neither $N(u)\subseteq N(w)$ nor $N(w)\subseteq N(u)$, then $G$ contains 3-sun as an induced subgraph.
\end{lemma}

\begin{proof}
By Lemma \ref{unilemma}, $G$ has a simplicial vertex $v$. If $u$ and $w$ be two vertices in $N(v)$ such that $N(u)\nsubseteq N(w)$ and $N(w)\nsubseteq N(u)$, then there exists vertices $u_1\in N(u)$ and $w_1\in N(w)$ in $G$ such that $u\nleftrightarrow w_1$ and $w\nleftrightarrow u_1$. For vertices $v_1$ and $v_2$ in $G$, let $P(v_1, v_2)$ denote an induced path connecting $v_1$ and $v_2$ in $G$. Since $G$ is a block, there exists a path $P(u_1, w_1)$ connecting $u_1$ and $w_1$ in $G$ not containing $v$. Since $G$ is a block we can choose a $P(u_1, w_1)$ such that $u, w \notin P(u_1, w_1)$. Let $t_1$ be the last vertex in $P(u_1, w_1)$ such that $u\leftrightarrow t_1$. Let $t_2$ be the last vertex in $P(u_1, w_1)$ such that $w\leftrightarrow t_2$. We consider the following two cases:
\begin{itemize}
    \item[Case 1:] $t_1\neq t_2$.\\
    Then $ut_1\dots t_2 w u$ is an induced cycle of length at least 4, a contradiction to $G$ being triangulated. Therefore Case 1 does not arise at all.
   
    \item[Case 2:] $t_1= t_2=t$, say.\\
    Then $t\leftrightarrow \{u,w\}$, and $G[\{u,u_1, \dots, t\}]$ will not contain an induced cycle of length at least 4, since $G$ is triangulated. Therefore $u$ is adjacent to all the vertices in $P(u_1,t)$. If $x$ is the vertex before $t_1$ in $P(u_1,t)$, then $x\leftrightarrow{u,t}$ and $x\nleftrightarrow w$, by the choice of $t$. Similarly there is a vertex $y$ in $P(w,t)$ such that $y\leftrightarrow \{w,t\}$ and $y\nleftrightarrow u$. Then $G[\{v,u,w,x,t,y\}]$=3-sun, a contradiction to our assumption. Therefore Case 2 does not arise at all. 
    \end{itemize}
    Hence the lemma.
  \end{proof}

\begin{lemma}\label{l1}
Let $G$ be a minimal $S$-imperfect graph. Then there exist no three distinct vertices $v, u, w\in V(G)$ such that $v$ is simplicial vertex in $G$ and $u$ and $w\in N(v)$, $N(u)\subseteq N(w)$ or $N(w)\subseteq N(u)$.
\end{lemma}
The proof is similar to the proof of Lemma \ref{neighbor} (We have to replace $v$ by $w$ in proof of Lemma \ref{neighbor}).

\begin{theorem}[\textbf{Main Theorem}]\label{main}
 A triangulated graph $G$ is $S$-perfect if and only if $G$ is odd sun-free.
\end{theorem}
\begin{proof}
\par Let $G$ be a triangulated $S$-perfect graph. Then by Lemma \ref{oddsunlemma} $G$ is odd-sun free.
\par Conversely, let $G$ be an odd sun-free triangulated graph.  If $G$ is not $S$-perfect, without loss of generality, let $G$ be a minimal $S$-imperfect graph. Then by Lemmas \ref{implemma} and \ref{lemma2.9}, there exists an extended sun $G^*$ such that $G$ is an induced subgraph of $G^*$ and $\alpha_S(G^*)\neq \theta_S(G^*)$. We may assume that $G^*$ is minimal $S$-imperfect. But then $G^*=G$, since $G$ is an induced subgraph of $G^*$. By Lemma \ref{l1}, $G$ has no three vertices $u,v,w\in V(G)$ such that $N(u)\subseteq N(w)$ or $N(w)\subseteq N(u)$, where $v$ is simplicial vertex in $G$ and $u,w\in N(v)$. Then $G$ contains 3-sun as an induced subgraph by Lemma \ref{mainlemma}. Since $G=G^*$, $G$ is an extended sun and contains 3-sun as an induced subgraph, a contradiction to our assumption. 
By definition of odd extended sun, it contains odd sun as an induced subgraph. Therefore $G$ is 3-sun free even extended sun. Then by Lemma \ref{evensunlemma}, $G$ is $S$-perfect. This completes the proof of the theorem.
\end{proof}

\section{Future Directions}
The main theorem leads to the following conjecture.
\par Conjecture. A graph is $S$-perfect if and only if $G$ is $\{C_{3k+1}, C_{3k+2}, k\geq 1$, odd super sun\}-free where super sun is defined as follows.
\par Let $H$ be a graph with a hamiltonian cycle $C=v_1v_2\dots v_k v_1$. Let $A_i, 1\leq i\leq k$ be mutually disjoint set of vertices not in $H$ such that $|A_i|\geq 1 $. Let $H^*$ be a graph constructed from $H$ such that 
$H^*$ has the following properties:
\begin{enumerate}
    \item the induced subgraph on $A_i$ is a path $P_i$ in $H^*$
    \item the end vertices of $P_i$, that is $u_{i1}$ and $u_{i|P_i|}$ are respectively adjacent to $v_i$ and $v_{i+1}$ and the number of vertices in $P_i$ is $3k+1, k\geq 1$
\end{enumerate}
Then $H^*$ is called $k$-super sun. $H^*$ is even or odd super sun depending on $k$ being even or odd.
\begin{figure}[h]
    \centering
    \includegraphics[width=0.21\linewidth]{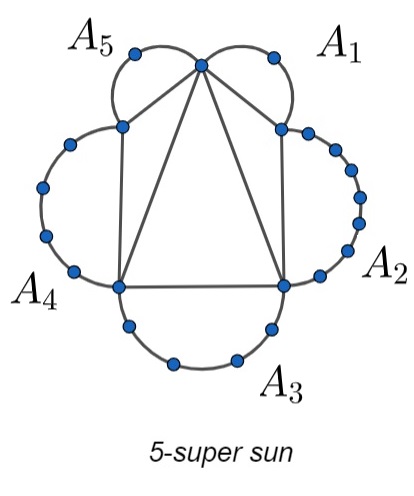}
    \caption{An example of 5-super sun graph}
    \label{}
\end{figure}
\newpage
\section*{Acknowledgement} 
The authors profusely thank S. A. Choudum, CHRIST (Deemed to be University) for helpful discussions and critically looking into the entire manuscript. 

\end{document}